\documentclass[a4paper,12pt,reqno]{amsart}
\usepackage{latexsym,amscd,amssymb,amsmath,amsthm,comment,url}
\usepackage[dvips]{graphicx}
\usepackage[all]{xy}
\xyoption{poly}
\xyoption{knot}
\xyoption{import}
\makeatletter
\theoremstyle{plain}
  \newtheorem{thm}{Theorem}[section]
  \newtheorem{prop}[thm]{Proposition}
  \newtheorem{lem}[thm]{Lemma}
  \newtheorem{cor}[thm]{Corollary}
  
\theoremstyle{definition}
  \newtheorem{dfn}[thm]{Definition}
  \newtheorem{exmp}[thm]{Example}
\theoremstyle{remark}
  \newtheorem{rem}[thm]{Remark}
\let\opn\operatorname %abbreviation of \operatorname
\let\term\emph
\def\@bothmode#1{\ifmmode #1\else $#1$\fi}
\newcount\@tempchn %new counter
\def\@chCount#1{%
   \@tempchn=0
   \@tfor\member:=#1\do{\advance\@tempchn by 1}%
}
\def\@autopr#1{%
   \@chCount{#1}%
   \ifnum\@tempchn<2 #1\else (#1)\fi
}
\let\@tempopn\relax %temporary operator
\def\@opform_#1#2{\@tempopn_{#1}\@autopr{#2}}
\numberwithin{equation}{section}
\def\NN{\mathbb{N}} %the monoid of natural numbers with 0
\def\ZZ{\mathbb{Z}} %the ring of integers
\def\RR{\mathbb{R}} %the field of real numbers
\def\kk{\Bbbk} %base field
\def\m{\ideal{m}} %maximal ideal
\def\p{\ideal{p}} %prime ideal
\let\e\varepsilon %incidence function
\let\s\sigma %cell
\let\C\Sigma %conical complex
\let\t\tau %cell
\let\u\upsilon %cell
\let\i\iota %injection
\def\MM{\mathcal M} %monoidal complex
\def\M{\mathbf M} %affine monoid
\let\@tempar\relax %temporary arrow command
\def\@seton^#1{\overset{#1}{\@tempar}}
\def\defar#1#2{\@xp\def\csname #1\endcsname{\def\@tempar{#2}\@ifnextchar^{\@seton}{\@tempar}}}
\defar{longto}{\longrightarrow} %\longto=\longrightarrow
\defar{epito}{\twoheadrightarrow} %\epito=\twoheadrightarrow
\defar{monoto}{\rightarrowtail} %\monoto=\rightarrowtail
\defar{embto}{\hookrightarrow} %\embto=\hookrighrarrow
\def\imply{\@bothmode{\Rightarrow}} % implication to the right
\def\Imply{\@bothmode{\Longrightarrow}} % long arrow version of \imply
\def\iff{\@bothmode\Longleftrightarrow} % equivalence
\def\get{\@bothmode{\Leftarrow}} % implication to the left
\def\Get{\@bothmode{\Longleftarrow}} % long arrow version of \get
\def\bra#1{[#1]} % square bracket
\def\mbra#1{\{ #1\}} % middle bracket
\def\set#1#2{\mbra{\,#1\mid #2\,}} %set
\let\Dsum\bigoplus %\big direct sum
\let\tns\otimes %small tensor product
\def\sM{|\MM |} %support of \M
\def\zM{\ZZ\MM} %ZM
\def\szM{|\ZZ\MM |} %support of \ZZ\M
\def\supp{\opn{supp}} %support
\def\geom#1{||#1||} %geometric realization
\let\none\varnothing %empty set (note that \empty is already defined for another use.)
\def\bc{a} %content
\def\cell{\mathcal X} %cell complex
\def\op{\mathsf{op}} %opposite
\def\chara{\operatorname{char}} %characteristic
\def\sph{\mathbb S} %sphere
\def\defopn#1{%
    \@xp\def\csname #1\endcsname{%
        \def\@tempopn{\opn{\csname the#1\endcsname}}%
        \@ifnextchar_{\@opform}{\@opform_{}}%
    }%
}

\defopn{Ker} %kernel 

\defopn{Im} %image

\defopn{Cok} %cokernel

\defopn{Ann} %annihilator

\defopn{Ass} %associated prime ideals

\defopn{Spec} %prime spectrum

\defopn{pd} %injective dimension

\defopn{id} %projective dimension
\def\idmap{\opn{id}} %\identity map
\def\nat{\opn{nat}} %\natural map
\def\the@init{in}
\defopn{@init}
\def\init{\@init_{\succ}}
\def\relint{\opn{rel-int}} %relative interior
\def\fring#1{\kk \bra{#1}} %face ring
\let\ideal\mathfrak %the font of ideals
\let\rad\sqrt %radical ideal
\def\theE{E}
\def\E{\@ifstar{{}^*\theE}{\theE}} %injective hull
\let\W\omega %canonical module
\let\defcat\defopn
\def\theMod{Mod}    \def\themod{mod}
      
\let\the@Lgr\theMod  \let\the@lgr\themod
\defcat{Mod} %the category of modules
\defcat{mod} %the category of f.g. modules
\defcat{@Lgr}
\def\Lgr#1{\@Lgr_{\ZZ\MM}{#1}} %the category of \ZM-graded modules
\defcat{@lgr}
\def\lgr#1{\@lgr_{\ZZ\MM}{#1}} %the category of f.g. \ZM-graded modules
\defcat{Sq} %the category of squarefree modules
\defcat{InjSq} %the category of injective objects of \Sq
\let\colimit\varinjlim %colimit (direct limit)
\def\theHom{Hom}    \def\theRHom{RHom}
\def\theExt{Ext}    
\def\theD{D}
\let\@tempgrop\underline
\def\Hom{\@ifstar{\opn{\@tempgrop\theHom}}{\opn\theHom}} %Hom functor
\def\RHom{\@ifstar{\opn{R\@tempgrop\theHom}}{\opn\theRHom}} %right derived functor of Hom
\def\Ext{\@ifstar{\opn{\@tempgrop\theExt}}{\opn\theExt}} %Ext functor
\def\uExt{\@ifstar{\opn{\@tempgrop\theuExt}}{\opn\theExt}} %Ext functor
\def\uExt{\underline{\operatorname{Ext}}}

\def\theDcat{D}
\def\Db{\theDcat^b} %bounded derived category
\def\@G_#1{\Gamma_{#1}}
\def\G{\@ifnextchar_{\@G}{\@G_\m}} %local cohomology functor
\def\RGm{\mathrm{R}\G} %right derived functor of \G
\def\DD{\mathbb D} %duality functor on Sq
\def\DDD{\mathbf D}
\def\theHtcat{K}
\def\Cb{\theHtcat^b} %homotopy category

\def\for{{\mathbb U}} %forgetful functor
\def\cpx#1{#1^{\bullet}} %cochain complexes
\def\theD{D} 
\def\D{\@ifstar{{}^*\theD}{\theD}} %dualizing complexes
\def\Sh{\operatorname{Sh}} %the category of sheaves on X
\def\cF{{\mathcal F}} %sheaf
\def\sp{\operatorname{Sp\acute{e}}} %etale space 
\def\RcHom{{\rm R}\operatorname{\mathcal Hom}} %right derived functor of sheaf hom
\def\cDx{\cpx {\mathcal D}_X} %dualizing complex of X
\def\const{\underline{\kk} } %constant sheaf with coefficients \kk
\def\mdL{\mod \Lambda}
\def\cH{\mathcal H} %cohomology of a complex of sheaves
\def\rH{\tilde{H}} %reduced homology
\makeatother
\title{Dualizing complex of a toric face ring}
\author{Ryota Okazaki}
\address{Department of Pure and Applied Mathematics, 
Graduate School of Information Science and Technology,
Osaka University, Toyonaka, Osaka 
560-0043, Japan}
\email{u574021d@ecs.cmc.osaka-u.ac.jp}
\author{Kohji Yanagawa}
\thanks{The second author is partially supported by Grant-in-Aid for Scientific Research (c) (no.19540028).}
\address{Department of Mathematics,
Faculty of Engeneering Science, Kansai University,
Suita 564-8680, Japan}
\email{yanagawa@ipcku.kansai-u.ac.jp}
\begin{document}
%
%---title---
%
\maketitle
%
%
%---body---
%
\begin{abstract} 
A {\it toric face ring}, which  generalizes both  Stanley-Reisner rings and 
affine semigroup rings, is studied by Bruns, R\"omer and their coauthors recently. 
In this paper, under the ``normality" assumption, we describe a dualizing complex of 
a toric face ring $R$ in a very concise way. 
Since $R$ is not a graded ring in general, 
the proof is not straightforward.   
We also develop the squarefree module theory over $R$, and   
show that the Buchsbaum property and the Gorenstein* property of $R$ 
are topological properties of its associated cell complex. 
\end{abstract}

\section{Introduction}
Stanley-Reisner rings and (normal) affine semigroup rings are important subjects of 
combinatorial commutative algebra. 
The notion of {\it toric face rings}, which originated in Stanley \cite{St}, 
generalizes both of them, and has been 
studied by Bruns, R\"omer, and their coauthors recently 
(e.g. \cite{BG02, BKR, IR}). 
%In this paper,  under the ``normality" assumption, 
%we construct a small complex which is quasi-isomorphic to a dualizing complex 
%of a toric face ring $R$. 
Contrary to Stanley-Reisner rings and affine semigroup rings,  
a toric face ring does not admit a nice multi-grading in general. 
So, even if the results can be easily imagined from these classical examples, 
the proofs sometimes require technical argument. 

\smallskip  

Now we start the definition of a toric face ring.   %This notion looks technical at first,  
%but the reader will find that it is natural. 
Let $\cell$ be a finite cell complex with $\emptyset \in \cell$. 
Assume that the closure $\overline{\s}$ of each $i$-cell $\s \in \cell$ is 
homeomorphic to an $i$-dimensional ball, and 
for given two cells $\s, \t \in \cell$ there exists $\u \in \cell$ with  
$\overline{\s} \cap \overline{\t} = \overline{\u}$ (we allow the case $\u = \emptyset$). 
A simplicial complex and the cell complex associated with a polytope are %preliminary 
examples of our $\cell$.

We assign a pointed polyhedral cone 
$C_\s \subset \RR^{d_\s}$ to each $\s \in \cell$ so that 
the following condition is satisfied.  (We say a cone is pointed if it contains no line.)
%(We say a cone is \term{pointed} if it contains no line.)  
\begin{itemize}
\item[$(*)$]  $\dim C_\s = \dim \s +1$, and 
there is a one-to-one correspondence between 
$\{ \, \text{faces of $C_\s$} \, \}$ 
and $\{ \, \t \in \cell \mid \t \subset \overline{\s} \, \}$. 
The face of $C_\s$ corresponding to $\t$ is isomorphic to $C_\t$ by a map 
$\i_{\s, \t}: C_\t \to C_\s$. These maps satisfy 
$\i_{\s, \s} = \idmap_{C_{\s}}$ and $\i_{\s, \t} \circ \i_{\t,\u}$ = $\i_{\s, \u}$ 
for all $\s,\t,\u \in \cell$ with $\overline{\s} \supset \overline{\t} \supset \u$.
\end{itemize}

For example, a pointed fan (i.e., a fan consisting of pointed cones) gives such a structure.  
Here $\i_{\s, \t}$'s are inclusion maps, and  
 $\cell$ is a ``cross-section'' of the fan. 
%For example, if the fan is complete, $\cell$ is a cell decomposition of an 
%$(n-1)$-dimensional sphere. 

Next we define a {\it monoidal complex} $\MM$ supported by $\{C_\s\}_{\s \in \cell}$ as follows. 
\begin{itemize}
\item[($**$)]To each $\s \in \cell$, we assign a finitely generated additive submonoid 
$\M_\s \subset (\ZZ^{d_\s} \cap C_\s) \subset \RR^{d_\s}$ with $\RR_{\geq 0} \M_\s = C_\s$. 
%(Hence the cone $C_\s$ is rational.) 
For $\s,\t \in \cell$ with $\overline{\s} \supset \t$, 
the map $\i_{\s,\t}:C_\t \to C_\s$ induces an isomorphism
$\M_\t \cong \M_\s \cap \i_{\s,\t}(C_\t)$ of monoids.
\end{itemize}

If $\Sigma$ is a rational pointed fan in $\RR^n$, then 
$\{ \, \ZZ^n \cap C \}_{C \in \Sigma}$ gives a monoidal complex. 

For a monoidal complex $\MM$ on a cell complex $\cell$, we set
$\sM := \colimit_{\s \in \cell}
\M_\s,$
where the direct limit is taken with respect to $\i_{\s, \t} : \M_\t \to \M_\s$ for 
$\s , \t \in \cell$ with $\overline{\s} \supset \t$. 
If $\MM$ comes from a fan in $\RR^n$, then $\sM$ can be identified with 
$\bigcup_{\s \in \cell} \M_\s \subset \ZZ^n$.   The $\kk$-vector space
$$
\fring \MM := \Dsum_{a \in \sM} \kk \, t^a,
$$
with the multiplication
$$
t^a \cdot t^b = \begin{cases}
t^{a+b} & \text{if $ a,b \in \M_\s $ for some $ \s \in \cell $;}\\
0       & \text{otherwise,}
\end{cases}
$$
has a $\kk$-algebra structure. 
We call $\fring \MM$ the toric face ring of $\MM$.
If $\MM$ comes from a fan in $\RR^n$, then $\kk[\MM]$ has a natural $\ZZ^n$-grading.  
However, this is not true in general (cf. Example~\ref{sec:Moebius_ring} below). 

%If $\MM$ comes from a fan $\C$ in $\RR^n$, then $\kk[\MM]$ has a $\ZZ^n$-graded ring structure 
%(while $\sM$ is {\it not} a semigroup).  
%But, in general, $\kk[\MM]$ does not have such a nice multi-grading. 

\begin{exmp}\label{sec:exmp_intro} 
(1) Let $\Delta$ be a simplicial complex.  
Attaching the monoid $\NN^{i+1}$ to each $i$-dimensional face of $\Delta$, we get a monoidal complex 
$\MM$ on $\Delta$. In this case, $\kk[\MM]$ coincides with 
the Stanley-Reisner ring $\kk[\Delta]$.    
An affine semigroup ring is also a toric face ring corresponding 
to the case when $\cell$ has a unique maximal cell.

(2)  Let $\cell$  be a two-dimensional cell complex given by the boundary of a cube. 
%It has 6 two-dimensional cells.
Assigning normal semigroup rings of the form 
$\kk[x, y, z, w]/(xz - y w)$ to all two-dimensional cells, 
we get a toric face ring $\kk[\MM]$. 
This $\MM$ comes from a fan, and $\kk[\MM]$ has a $\ZZ^3$-grading with 
$\M_\s = \ZZ^3 \cap C_\s$ for all $\s \in \cell$. (Find such a grading explicitly.)  
Next, we assign  $\kk[x, y, z, w]/(xz - y w)$ to 5 two-dimensional cells
and $\kk[x,y,z,w,v]/ (xz-v^2, yw-v^2)$ to the $6^{\rm th}$ one. 
Then we get a toric face ring 
$\kk[\MM']$, which is observed in \cite[pp.6-7]{BG02}. 
While $\kk[\MM']$ admits a $\ZZ^3$-grading and all $\fring{\M'_\s}$ is normal, 
it is impossible to satisfy $\M'_\s = \ZZ^3 \cap C_\s$  simultaneously for all $\s$.   
A toric face ring without multi-grading is given in Example~\ref{sec:Moebius_ring}.
\begingroup
\def\A{\fring{x,y,z,w}/(xz-yw)}
\def\B{\fring{x,y,z,w,v}/(xz-v^2,yw-v^2)}    
\begin{figure}[h]
\begin{minipage}{0.3\hsize}
$$
\xy /r3.0pc/:,
{\xypolygon4"A"{\bullet}},
+(.5,.6),
{\xypolygon4"B"{~>{}\bullet}},
{"A2"\PATH~={**@{-}}'"B2"'"B1"'"B4"'"A4"},
{"A3"\PATH~={**@{--}}'"B3"'"B4"},
"A1";"B1"**@{-},
"B2";"B3"**@{--},
{"B1";"A4"**@{.}},?!{"A1";"B4"**@{.}},*{\circ}
\endxy
$$
\end{minipage}
\begin{minipage}{0.3\hsize}
$$
\xy
\xyimport(3,3)(0.5,0.5){}="a",
"a"+L; "a"+R, **{},
?(0); ?(1)*@{>}**@{-},
"a"+D; "a"+U, **{},
?(0); ?(1)*@{>}**@{-},
(0,1.5)="A"; (1.5,1.5)="C"**@{-},
(1.5,0)="B"; "C"**@{-},
0*{\bullet}, "A"*{\bullet}, "B"*{\bullet}, "C"*{\bullet},
(1.2,-1)*{\A}
\endxy
$$
\end{minipage}
\begin{minipage}{0.3\hsize}
$$
\xy
\xyimport(3,3)(1.5,1.5){}="a",
"a"+L; "a"+R, **{},
?(0); ?(1)*@{>}**@{-},
"a"+D; "a"+U, **{},
?(0); ?(1)*@{>}**@{-},
(0,1)="A"; (-1,0)="B"**@{-},
"B"; (0,-1)="C"**@{-},
"C"; (1,0)="D"**@{-},
"D"; "A"**@{-},
"A"*{\bullet}, "B"*{\bullet}, "C"*{\bullet}, "D"*{\bullet},  0*\cir<2.6pt>{},
(0.7,-2)*{\B}
\endxy
$$
\end{minipage}
%\caption{The cube and its faces}
\end{figure}
\endgroup
\end{exmp}

The affine semigroup ring  $\kk[\M_\s] := \bigoplus_{a \in \M_\s} \kk \, t^a$ 
can be regarded as a quotient ring of a toric face ring $R:=\kk[\MM]$.  
In the rest of this section, 
we assume that $\kk[\M_\s]$ is normal for all $\s \in \cell$, 
and set $d:=\dim R = \dim \cell +1$.  
%The following is a main result of this paper. 

\begin{thm}\label{Ishida}
In the above situation, the cochain complex $\cpx I_R$ 
given by 
$$I^{-i}_R := \Dsum_{\substack{\s \in \cell, \\ \dim \s = i-1}} \kk[\M_\s],  \qquad 
\cpx I_R: 0 \longto I^{-d}_R \longto I^{-d + 1}_R \longto \cdots \longto I^0_R \longto 0, 
$$
and $$\partial: I_R^{-i} \supset \fring {\M_\s} \ni 1_\s \longmapsto 
\sum_{\substack{\dim \fring{\t} = i-1, \\ \t \subset \overline{\s}}} \pm 1_\t \in 
\Dsum_{\substack{\dim \fring{\t} = i-1, \\ \t \subset \overline{\s}}} \fring 
{\M_\t}  \subset I_R^{-i+1}$$ 
is quasi-isomorphic to a normalized dualizing complex $\cpx D_R$ of $R$. 
Here the sign $\pm$ is given by an incidence function of the regular cell complex $\cell$.   
\end{thm}

Clearly, our $\cpx I_R$ is analogous to the complex constructed in Ishida \cite{I},  
but, since we assume that all $\kk[\M_\s]$ are normal,  
we do not have to take the (graded) injective hull of $\kk[\M_\s]$.  
If $\MM$ comes from a fan in $\RR^n$, the above theorem has been obtained in
\cite[Theorem~5.1]{IR} using the $\ZZ^n$-grading of $R$.

We also introduce the notion of $\ZZ\MM$-graded $R$-modules.   Since $R$ is not a graded ring, 
these are not graded modules in the usual sense, but we can consider their ``Hilbert functions". 
In particular, Corollary~\ref{reduced cohomology}, which recaptures a result of \cite{BBR}, gives 
a formula on the Hilbert function of the local cohomology module $H_\m^i(R)$ at the maximal ideal 
$\m := (\, t^a \mid 0 \ne a \in \sM \, )$. 

In \cite{Y01,Y03}, the second author defined {\it squarefree modules} $M$ over a normal semigroup ring 
$\kk[\M_\s]$, and gave corresponding constructible sheaves $M^+$ on the closed ball $\overline{\sigma}$.  
We can extend this to a toric face ring $R$, that is, we define squarefree $R$-modules and 
associate constructible sheaves on $\cell$ with them. 
In this context, the duality $\RHom_R(-, \cpx I_R)$ on the derived category of 
squarefree $R$-modules corresponds to Poincar\'e-Verdier duality on 
the derived category of constructible sheaves on $\cell$. 
For example, the complex $\cpx I_R$  consists of squarefree modules, and 
$(\cpx I_R)^+$ is the Verdier's dualizing complex of the underlying topological space of $\cell$. % with the coefficients in $\kk$. 

\begin{cor}\label{sec:Intr2}
The Buchsbaum property, Cohen-Macaulay property and Gorenstein* property 
are topological properties of the underlying space of $\cell$. % (while it may depend on $\chara(\kk)$).
\end{cor}

For instance, both $\kk[\MM]$ and $\kk[\MM']$ of Example~\ref{sec:exmp_intro} (b) are Gorenstein*.
While some parts/cases of Corollary~\ref{sec:Intr2} have been obtained in existing papers, 
our argument gives systematic perspective.

\section{Toric face rings}
First, we shall recall the definition of a regular cell complex:
A \term{ finite regular cell complex} (cf. \cite[Section 6.2]{BH}) is a topological space $X$ 
together with a finite set $\cell$ of subsets of $X$ such that the following 
conditions are satisfied: 
\begin{enumerate}
\item $\emptyset \in \cell $ and 
$X = \bigcup_{\sigma \in \cell } \sigma$; 
\item the subsets $\sigma \in \cell $ are pairwise disjoint;
\item for each $\sigma \in \cell $, $\sigma \ne \emptyset$, 
there exists some $i \in \NN$ and a homeomorphism from an $i$-dimensional ball 
$\{ x \in \RR^i \mid ||x|| \leq 1 \}$ to the closure 
$\overline{\sigma}$ of $\sigma$ which maps $\{ x \in \RR^i \mid ||x|| < 1 \}$ onto $\sigma$.  
\item For any $\sigma \in \cell$, the closure 
$\overline{\sigma}$ can be written as the union of some cells in $\cell$. 
\end{enumerate}  

An element $\sigma \in \cell$ is called a {\it cell}. 
We regard $\cell$ as a poset with the order $>$ defined as follows; $\sigma \ge \tau$ 
if $\overline{\sigma} \supset \tau$.  
If $\overline{\sigma}$ is homeomorphic to an $i$-dimensional ball, 
we set $\dim \sigma = i$.  Here $\dim \emptyset = -1$. 
Set $\dim X = \dim \cell := \max \{ \, \dim \sigma \mid \sigma \in  \cell  \, \}$. 

Let $\s, \tau \in \cell $. If $\dim \s =i+1$, 
$\dim \tau = i-1$ and $\tau < \s$, 
then there are exactly two cells $\s_1, \s_2 \in \cell $ between  
$\tau$ and $\s$. (Here $\dim \s_1 = \dim \s_2 = i$.)
A remarkable property of a regular cell complex is the existence of an 
{\it incidence function} $\varepsilon$ satisfying the following conditions. 
\begin{enumerate}
\item To each pair $(\s, \tau)$ of cells, $\varepsilon$ assigns a 
number $\varepsilon(\s,\tau) \in \{0, \pm 1\}$. 
\item $\varepsilon(\s,\tau) \ne 0$ if and only if
$\dim \tau = \dim \s -1$ and $\tau < \s$.  
%\item[(iii)] If $\dim \s = 0$, then $\varepsilon(\sigma,\emptyset) =1$. 
\item If $\dim \s =i+1$, $\dim \tau = i-1$ and 
$\tau < \s_1, \, \s_2 < \s$, $\sigma_1 \ne \sigma_2$, 
then we have 
$$\varepsilon(\s, \s_1) \, \varepsilon(\s_1, \tau) + 
\varepsilon(\s, \s_2) \, \varepsilon(\s_2, \tau) =0.$$
\end{enumerate}
We can compute the (co)homology groups of $X$ using the cell decomposition 
$\cell$ and an incidence function $\varepsilon$. 

\begin{exmp}\label{sec:exmp_cell_cpx}
We shall give two typical examples of a finite regular cell complex: one is associated with a simplicial complex $\Delta$ 
on the vertex set $\bra n := \mbra{1,\dots ,n}$, i.e., a subset of the power set $2^{\bra n}$ such that,
for $F, G \in 2^{\bra n}$, $F \subset G$ and $G \in \Delta$ imply $F \in \Delta$. Take its geometric realization $\geom \Delta$, and
let $\rho$ be the map giving the realization (see \cite{BH} for the definition of a geometric realization). Then $X := \geom \Delta$ together with
$\set{\relint(\rho(F))}{F \in \Delta}$ is a regular cell complex, where $\relint(\rho(F))$ denotes the relative interior of $\rho(F)$.

The other example is a polytope $P$. In this case, $P$ itself is the underlying topological space; 
the cells are the relative interiors of its faces.
\end{exmp}

\begin{dfn}\label{sec:cell_cpx_ver}
A \term{conical complex} consists of the following data. 
\begin{enumerate}
\item A finite regular cell complex $\cell$ satisfying the intersection property, i.e.,
for $\s,\t \in \cell$, there is a cell $\u \in \cell$ such that $\overline \u = \overline \s \cap \overline \t$;
\item A set $\C$ of finitely generated cones $C_{\s} \subset \RR^{\dim \s + 1}$
with $\s \in \cell$ and $\dim C_{\s} = \dim \s + 1$. 
\item An injection $\i_{\s,\t}:C_{\t} \to C_{\s}$ for $\s, \t \in \cell$ with $\s \ge \t$ satisfying the following. 
\begin{enumerate}
\item $\i_{\s,\t}$ can be lifted up to a linear map  $\RR^{\dim \t+1} \to \RR^{\dim \s +1}$. 
\item  The image $\i_{\s, \t}(C_{\t})$ is a face of $C_{\s}$.  
Conversely, for a face $C'$ of $C_\s$, there is a sole cell $\t$ with $\t \le \s$
such that $\i_{\s,\t}(C_\t) = C'$. 
Thus we have a one-to-one correspondence between 
$\{ \, \text{faces of $C_\s$} \, \}$ 
and $\{ \, \t \in \cell \mid \t \le \s \, \}$. 
\item $\i_{\s, \s} = \idmap_{C_{\s}}$ and $\i_{\s,\t} \circ \i_{\t,\u}$ = $\i_{\s,\u}$ for $\s,\t,\u \in \cell$
with $\s \ge \t \ge \u$.
\end{enumerate}
\end{enumerate}
We denote this structure by $(\C, \cell)$ or $\C$ simply.
\end{dfn}

\begin{rem}
(1) We have $\none \in \cell$ according to the definition of a regular cell complex, and
the corresponding cone $C_\none$ is $\{ 0 \}$.
Thus for a conical complex $(\C,\cell)$, each $C_\s \in \C$ is \term{pointed}, i.e., $\{ 0 \}$ 
is a face of $C_\s$.

(2) The concept of conical complexes was first defined by Bruns-Koch-R\"omer \cite{BKR} in a slightly different manner, but, under the additional condition that each cone is pointed, 
their definition is equivalent to ours. 
That is, our conical complexes are {\it pointed} conical complexes of \cite{BKR}. 
\end{rem}

For grasping the image of a conical complex $(\C,\cell)$,
it is helpful to regard the conical complex as
the object given by ``gluing" each cones along the injections $\i_{\s, \t}$.
A typical example of a conical complex is a pointed fan,
i.e., a finite collection $\C$ of pointed cones in $\RR^n$ satisfying the following properties:
\begin{enumerate}
\item for $C' \subset C \in \C$, $C'$ is a face of $C$ if and only if $C' \in \C$;
\item for $C,C' \in \C$, $C \cap C'$ is a common face of $C$ and $C'$.
\end{enumerate}
In this case, as an underlying cell complex, we can take 
$\set{\relint(C \cap \sph^{n-1})}{C \in \C}$, where $\sph^{n-1}$ denotes the unit sphere in 
$\RR^n$, and the injections $\i$ are inclusion maps.

\begin{exmp}
There exists a conical complex which is not a fan. In fact, consider the M\"obius strip
as follows.
\begin{figure}[h]\label{sec:Moebius}
$$
\xy /r2.5pc/:,
{\xypolygon3"A"{~={75}~:{(-1,1.7)::}~>{}\bullet}},
+(.8,.8),
{\xypolygon3"B"{~={75}~:{(-1,1.7)::}~>{}\bullet}},
{"A1"\PATH~={**@{-}}'"A2"'"A3"'"B3"'"B2"'"B1"'"A1"},
"A2";"B2"**@{-},
{\vtwist~{"A1"}{"B1"}{"A3"}{"B3"}},
"A1"*+!RD{x}, "A2"*+!R{y}, "A3"*+!LU{z},
"B1"*+!RD{u}, "B2"*+!R{v}, "B3"*+!LU{w}
\endxy
$$
%\caption{M\"obius strip}
\end{figure}
\noindent Regarding each rectangles as the cross-sections of 3-dimensional cones,
we have a conical complex that is not a fan (see \cite{BG}).
\end{exmp}

A monoidal complex plays a role similar to the defining semigroup of an affine semigroup ring.

\begin{dfn}[\cite{BKR}]
A \term{monoidal complex} $\MM$ supported by a conical complex $(\C,\cell)$ is a set of monoids
$\mbra{\M_\s}_{\s \in \cell}$ with the following conditions:
\begin{enumerate}
\item $\M_\s \subset \ZZ^{\dim \s+1}$ for each $\s\in \cell$, 
and it is a finitely generated additive submonoid (so $\M_\s$ is an affine semigroup);
\item $\M_\s \subset C_\s$ and $\RR_{\geq 0} \M_\s = C_\s$ for each $\s\in \cell$ 
(hence the cone $C_\s$ is automatically rational); 
\item for $\s,\t \in \cell$ with $\s \ge \t$, the map $\i_{\s,\t}:C_\t \to C_\s$ induces an isomorphism
$\M_\t \cong \M_\s \cap \i_{\s,\t}(C_\t)$ of monoids.
\end{enumerate}
\end{dfn}

For example, let $\C$ be a rational pointed fan in $\RR^n$. Then $\set{C \cap \ZZ^{n}} {C \in \C}$ gives 
a monoidal complex. More generally, a family of affine semigroups 
$\set{\M_C \subset \ZZ^n}{C \in \C}$ satisfying the following conditions, forms a monoidal complex;
\begin{enumerate}
\item $\RR_{\geq 0} \M_C = C$ for each $C \in \C$;
\item $\M_C \cap C' = \M_{C'}$ for $C, C' \in \C$ with $C' \subset C$.
\end{enumerate}

\begin{rem}\label{sec:rem_on_ccpx}
(1) In \cite[\S2]{BG02}, basic properties of a {\it rational polyhedral complex}, 
which gives a conical complex and a monoidal complex in a natural way, 
are discussed. 

(2) Even if a regular cell complex $\cell$ satisfies the intersection property, 
there does not exist a conical complex of the form $(\C, \cell)$ in general. 
For example, there is a simplicial complex $\Delta$ such that 
the geometric realization $||\Delta||$ is homeomorphic to a 3-dimensional sphere,
but $\Delta$ is not the boundary complex of any (4-dimensional) polytope.  
See, for example,  \cite[Notes of Chap. 8]{Z}.
Now take a $4$-dimensional ball, and let $\s$ be its interior.
Triangulating the boundary of the ball, which is a $3$-dimensional sphere,
according to $\Delta$, we obtain the cell complex $\cell := \Delta \cup \{ \s \}$
such that $\s > \tau$ for all $\tau \in \Delta$.
If there is a conical complex of the form $(\C, \cell)$, then the boundary complex of 
a cross section of the cone $C_\s \in \C$ coincides with $\Delta$. 
This is a contradiction.  

On the other hand, for any 2-dimensional regular cell complex $\cell$ satisfying the intersection property,
there is a conical complex $(\C,\cell)$
and a monoidal complex $\MM$ supported by it as follows.

Let $n \geq 3$ be an integer. It is an easy exercise to construct an affine semigroup 
$\M_n \subset \NN^3$ satisfying the following conditions. 
\begin{itemize}
\item[(i)] The cone $C:= \RR_{\geq 0} \M_n \subset \RR^3$ has exactly $n$ extremal rays, that is, 
its cross section is an $n$-gon.
\item[(ii)] For any 2-dimensional face $F$ of $C$, we have $F \cap \M_n \cong \NN^2$ as monoids.  
\end{itemize}
For a 2-dimensional cell $\s \in \cell$, set $n(\s) := \# \{ \tau \mid \tau \leq \s, \dim \tau =1 \}$. 
By the intersection property of $\cell$, we have $n(\s) \geq 3$. The assignment 
$\M_\s := \M_{n(\s)}$ for each 2-dimensional cell $\s$ gives a monoidal complex on $\cell$. 
\end{rem}

For a conical complex $(\C,\cell)$ and a monoidal complex $\MM$ supported by $\C$, we set
$$
\sM := \colimit_{\s \in \cell}\M_\s, \quad \szM := \colimit_{\s \in \cell}\ZZ \M_\s,
$$
where the direct limits are taken with respect to the inclusions $\i_{\s, \t} : \M_\t \to \M_\s$ and
induced map $\ZZ \M_\t \to \ZZ \M_\s$ respectively, for $\s, \t \in \cell$ with $\s \ge \t$.

Let $a, b \in \szM$. If there is some $\s \in \cell$ with $a, b \in \ZZ \M_\s$, 
by the intersection property of $\cell$, there is a unique minimal cell among 
these $\s$'s. Hence we can define $a \pm b \in \szM$. 
 
\begin{dfn}[\cite{BKR}]
Let $(\C,\cell)$ be a conical complex, $\MM$ a monoidal complex supported by $\C$, and $\kk$ a field.
Then the $\kk$-vector space
$$
\fring \MM := \Dsum_{a \in \sM} \kk \, t^a,
$$
where $t$ is a variable, equipped with the following multiplication
$$
t^a \cdot t^b = \begin{cases}
t^{a+b} & \text{if $ a,b \in \M_\s $ for some $ \s \in \cell $;}\\
0       & \text{otherwise,}
\end{cases}
$$
has a $\kk$-algebra structure. 
We call $\fring \MM$ the \term{toric face ring} of $\MM$ over $\kk$.
\end{dfn}

It is easy to see that $\dim R = \dim \cell +1$. 
When $\C$ is a rational pointed fan, $\fring \MM$ coincides with a toric face ring of
Ichim-R\"omer's sense (\cite{IR}). Moreover, if we choose $C_\s \cap \ZZ^n$ as $\M_\s$ for each $\s$,
$\fring \MM$ is just an earlier version due to Stanley (\cite{St}).
Henceforth we refer a toric face ring of $\MM$ supported by a fan as an \term{embedded} toric face ring.
Every Stanley-Reisner ring and every affine semigroup ring (associated
with a positive affine semigroup) can be established
as embedded toric face rings (see Example~\ref{sec:exmp_intro}).
The most difference between an embedded toric face ring and a non-embedded one, is
whether it has a nice $\ZZ^n$-grading or not;
an embedded toric face ring always has the natural $\ZZ^n$-grading such that the dimension,
as a $\kk$-vector space, of each homogeneous component is less than or equal to $1$.
However a non-embedded one does not have such a grading.

\medskip

Toric face rings can be expressed as a quotient ring of a polynomial ring.
Let $\MM$ be a monoidal complex supported by a conical complex $(\C,\cell)$, and $\mbra{a_e}_{e \in E}$  a family of elements
of $\sM$ generating $\fring \MM$ as a $\kk$-algebra, or equivalently, $\mbra{a_e}_{e \in E} \cap \M_\s$ 
generates $\M_\s$ for each $\s \in \cell$.
Then the polynomial ring $S := \kk[ \, X_e \mid e\in E \, ]$ surjects on $\fring \MM$. We denote, by $I_\MM$, its kernel.
Similarly we have the surjection $S_\s := \kk[ \, X_e \mid a_e \in \MM_\s, \ e \in E \, ] \epito \fring{\M_\s}$,
where $\fring{\M_\s}$ denotes the affine semigroup ring of $\M_\s$,
and denote its kernel by $I_{\M_\s}$.

\begin{prop}[{\cite[Proposition 2.6]{BKR}}]
With the above notation, we have
$$
I_\MM = A_\MM + \sum_{i = 1}^n SI_{\M_{\s_i}},
$$
where $\s_1, \dots , \s_n$ are the maximal cells of $\cell$, and $A_\MM$ is the ideal of $S$ generated by the squarefree monomials $\prod_{h \in H}X_h$ for which $\set{a_h}{h \in H}$
is not contained in $\M_\s$ for any $\s \in \cell$.
\end{prop}

\begin{exmp}[{\cite[Example 4.6]{BKR}}]\label{sec:Moebius_ring}
Consider the conical complex given in Example \ref{sec:Moebius}, and choose each rectangles to be a unit square.
In this case, we can construct a monoidal complex $\MM$ such that $\M_\s = C_\s \cap \ZZ^{\dim C_\s}$ for all $\s$,
and then $u,v,w,x,y,z$ are generators of $\MM$. We set $S := \fring{X_u,X_v,X_w,X_x,X_y,X_z}$,
where $X_u,\dots ,X_z$ are variables. Clearly, 
$\fring{\M_\s}$ is a polynomial ring if $\dim \s \le 1$, and one of the following
$$
\fring{X_u,X_v,X_x,X_y}/(X_xX_v - X_uX_y), \quad \fring{X_v,X_w,X_y,X_z}/(X_vX_z - X_yX_w)
$$
$$
\fring{X_u,X_w,X_x,X_z}/(X_xX_z - X_uX_w),
$$
if $\dim \s = 2$. Therefore we conclude that
$$
I_\MM = (X_xX_v - X_uX_y, X_vX_z - X_yX_w, X_xX_z - X_uX_w, X_uX_vX_w, X_uX_vX_z) \subset S.
$$
We leave the reader to verify that the other squarefree monomials in $A_\MM$, 
e.g. $X_xX_yX_z$, are indeed contained in the above ideal.
\end{exmp}

In this paper, we often assume that $\kk[\MM]$ satisfies the following condition. 

\begin{dfn}
We say a toric face ring $\kk[\MM]$  (or a monoidal complex $\MM$)  
is {\it cone-wise normal}, if the affine semigroup ring 
$\kk[\M_\s]$ is normal for all $\s \in \cell$. 
\end{dfn}

If $\kk[\MM]$ is cone-wise normal, then  $\kk[\M_\s]$ is Cohen-Macaulay for all $\s \in \cell$. 
Clearly, the toric face rings given in Examples~\ref{sec:exmp_intro} and \ref{sec:Moebius_ring} are cone-wise normal.  

\begin{rem}
The notion of a cone-wise normal monoidal complex $\MM$ 
is equivalent to that of the lattice points 
${\mathcal WF}(\Pi_{\rm rat})$ of a {\it weak fan} ${\mathcal WF}$ introduced by 
Bruns and Gubeladze in  \cite[Definition~2.6]{BG02}.  
In this case, our ring $\kk[\MM]$ is the same thing as the ring $\kk[{\mathcal WF}]$ 
of \cite{BG02}. 

An affine semigroup ring $A=\kk[\M_\s]$ has a graded ring structure 
$A = \bigoplus_{i \in \NN} A_i$ with $A_0 = \kk$. 
The toric face ring given in Example~\ref{sec:Moebius_ring} also has an $\NN$-grading 
given by   $\deg X_u = \cdots =\deg X_z = 1$. This is not true in general;
there is a monoidal complex whose toric face ring does not have an $\NN$-grading. See \cite[Example~2.7]{BG02}.
\end{rem}

For a commutative ring $A$, let $\Mod A$ (resp. $\mod A$) denote the category of 
(resp. finitely generated) $A$-modules.

\begin{dfn}
Let $R := \fring \MM$ be a toric face ring of a monoidal complex $\MM$ supported by
a conical complex $(\C,\cell)$.
\begin{enumerate}
\item $M \in \Mod R$ is said to be $\zM$-graded if the following conditions are satisfied;
  \begin{enumerate}
  \item $M = \Dsum_{a \in \szM}M_a$ as $\kk$-vector spaces;
  \item $t^a \cdot M_b \subset M_{a + b}$ if $a \in \M_\s$ and $b \in \ZZ\M_\s$ for some $\s \in \cell$,
and $t^a \cdot M_b = 0$ otherwise.
  \end{enumerate}
\item $M \in \Mod R$ is said to be $\MM$-graded if it is $\zM$-graded and $M_a = 0$ for $a \not\in \sM$.
\end{enumerate}
\end{dfn}

Of course, setting $R_a := \kk \, t^a$ for each $a \in \sM$, we see that $R$ itself is $\sM$-graded.
Any \term{monomial ideal}, i.e., an ideal generated by elements of the form $t^a$
for some $a \in \sM$, is $\MM$-graded, and hence $\zM$-graded. 
Conversely, every $\zM$-graded ideal is a monomial ideal.

Let $\Lgr R$ (resp. $\lgr R$) denote the subcategory of $\Mod R$ (resp. $\mod R$)
whose objects are $\zM$-graded
$R$-modules and morphisms are degree preserving maps, i.e., $R$-homomorphisms $f:M \to N$ such that
$f(M_a) \subset N_a$ for $a \in \szM$. It is clear that $\Lgr R$ and $\lgr R$ are abelian.

For each $\s \in \cell$,  the ideal $\p_\s := (t^a \mid a \not\in \M_\s) \subset R$ is a $\zM$-graded prime ideal 
since $R/\p_\s \cong \fring{\M_\s}$.
Conversely, every $\zM$-graded prime ideals are of this form.

\begin{lem}\label{sec:prime_ideal}
There is a one-to-one correspondence between the cells in $\cell$ and the $\zM$-graded prime ideals of $R$.
$$
\xymatrix@R=0pt{
\cell \ar@{<->}[r] & \left\{\parbox{9em}{\centering all the $ \zM $-graded
prime ideals of $R$}\right\} \\
\rotatebox{90}{$ \in $} & \rotatebox{90}{$ \in $} \\
\s \ar@{<->}[r] & \p_\s
}
$$
\end{lem}
\begin{proof}
The proof is quite the same as \cite[Lemma 2.1]{IR}.
\end{proof}

For an ideal $I$ of $R$, we denote, by $I^*$, the ideal of $R$
generated by all the monomials belonging to $I$.
As in the case of a usual grading, we have the following:

\begin{lem}\label{sec:graded_prime}
For a prime ideal $\p$ of $R$, $\p^*$ is also prime, and hence is a $\zM$-graded prime ideal.
\end{lem}
\begin{proof}
Since the ideal $0$ can be decomposed as follows
$$
\bigcap_{\substack{\s \in \cell \\ \s\text{:maximal}}}\p_\s = 0,
$$
$\{ \, \p_\s \mid \text{$\s$ is a maximal cell of $\cell$} \, \}$ is the set of minimal primes of $R$. 
Hence $\p$ must contain $\p_\s$ for some $\s \in \cell$. It follows that $\p^* \supset \p_\s$. 
Consider the images $\rho(\p)$ and $\rho(\p^*)$ by the surjection $\rho:R \epito \fring{\M_\s}$.
Then $\rho(\p)$ is prime and $\rho(\p^*)$ is the ideal generated by the monomials 
contained in $\rho(\p)$, whence is prime. Therefore we conclude that $\p^*$ is also prime.
\end{proof}

\begin{cor}\label{sec:radical}
Let $\ideal a$ be a $\zM$-graded ideal of $R$. Then its radical ideal $\rad{\ideal a}$ is also $\zM$-graded.
\end{cor}
\begin{proof}
Since $\ideal a \subset \p^*$ holds for a prime ideal $\p$ with $\ideal a \subset \p$,
we have
$$
\bigcap_{\p \supset \ideal a} \p^* \subset \bigcap_{\p \supset \ideal a} \p = \rad{\ideal a}
%\subset \bigcap_{\substack{\p \supset \ideal a \\ \p \text{:graded}}} \p
\subset \bigcap_{\p \supset \ideal a} \p^*,
$$
and therefore $\rad{\ideal a} = \bigcap_{\p \supset \ideal a} \p^*$.
\end{proof}

\section{C\v ech complexes and local cohomologies}\label{sec:Cech_cpx_and_local_cohom}

Let $(\C,\cell)$ be a conical complex, and $\MM$ a monoidal complex.
For $\s \in \cell$, set $T_\s := \set{t^a}{a \in \M_\s} \subset R := \fring \MM$.
Then $T_\s$ forms a multiplicatively closed subset consisting of monomials.
Moreover, a multiplicatively closet subset $T$ consisting of monomials is contained in 
some $T_\s$, unless $T \ni 0$.   

\begin{lem}\label{sec:localization}
Let $M \in \Lgr R$, and let $T$ be a multiplicatively closed subset of $R$ consisting of monomials.
Then $T^{-1}M \in \Lgr R$.
\end{lem}
\begin{proof}
Take any $x/t^a \in T^{-1}M$ with $a \in \sM$, $b \in \szM$, and $x \in M_b$. 
If there is no $\s \in \cell$ with $a, b \in \ZZ \M_\s$, then $x/t^a = (xt^a)/t^{2a} = 0$;
otherwise, $b-a$ is well-defined and in $\szM$. Now for $\lambda \in \szM$, set
$$
(T^{-1}M)_{\lambda} := \sum_{x \in M_b, b-a = \lambda}\kk\cdot\frac x{t^a}
$$
Then we have $T^{-1}M = \Dsum_{\lambda \in \szM}(T^{-1}M)_{\lambda}$ as $\kk$-vector spaces, which gives $T^{-1}M$
a $\szM$-grading.
\end{proof}

Well, set 
$$
L^i_R := \Dsum_{\substack{\s \in \cell \\ \dim \s = i-1}}T_\s^{-1}R
$$
and  define $\partial:L^i_R \to L^{i+1}_R$
by
$$
\partial(x) = \sum_{\substack{\t \ge \s \\ \dim \t = i}} \e(\t,\s) \cdot \nat(x)
$$
for $x \in T^{-1}_\s R \subset L^i_R$,
where $\e$ is an incidence function on $\cell$ and $\nat$ is a natural map $T^{-1}_\s R \to T^{-1}_\t R$
for $\s \le \t$.
Then $(\cpx L_R, \partial)$ forms a complex in $\Lgr R$:
$$
\cpx L_R: 0 \longto L^0_R \longto^\partial L^1_R \longto^\partial \cdots \longto^\partial 
L^d_R \longto 0,
$$
where $d= \dim R = \dim \cell +1$. 
We set $\m := (t^a \mid 0 \not= a \in \sM)$. This is a maximal ideal of $R$.

\begin{prop}[{cf. \cite[Theorem~4.2]{IR}}]
For any $R$-module $M$,
$$
H^i_\m(M) \cong H^i(\cpx L_R \tns_R M),
$$
for all $i$.
\end{prop}
\begin{proof}
%If $\dim \C = 0$, then $R = \kk$, and the assertion is trivial; we may assume that $\dim \C > 0$.
%Since $H^i_\m(-)$ is an $i^{\text{th}}$ right derived functor of $\G(-) := \colimit_n \Hom(R/\m^n,-)$,
It suffices to show the following:
\begin{enumerate}
\item $H^0(\cpx L_R \tns_R M) \cong H^0_\m(M)$;
\item for a short exact sequence $0 \to M_1 \to M_2 \to M_3 \to 0$ in $\Mod R$, the induced one
$0 \to \cpx L_R \tns_R M_1 \to \cpx L_R \tns_R M_2 \to \cpx L_R \tns_R M_3 \to 0$ is also exact;
\item for any injective $R$-module $I$, $H^i(\cpx L_R \tns_R I) = 0$ for all $i \ge 1$.
\end{enumerate}
Let $\ideal a$ be the ideal generated by elements $t^a$ with $0 \not= a \in C_\s$ 
for some $1$-dimensional cone $C_\s$.
%(The fact that $\dim \C > 0$ indeed guarantees the existence of such a cone $C$.)
Since $\Ker{L^0_R \tns_R M \to L^1_R \tns_R M} = H^0_{\ideal a}(M)$, to prove (1), we only have to show that
$\rad{\ideal a} = \m$. Let $\p$ be a prime containing $\ideal a$. Since $\ideal a$ is graded,
we have $\p^* \supset \ideal a$.
Thus there exists $\t \in \cell$ such that $\p_\t \supset \ideal a$,
but then $C_\t$ contains no $1$-dimensional face.
Therefore we conclude that $\p_\t = \p_\none = \m$, which implies $\rad{\ideal a} = \m$.

The condition (2) follows easily from the flatness of the localization. 
For (3), we can apply the same argument
of Ichim and R\"omer \cite{IR} for embedded toric face rings 
(but we need to use Lemma \ref{sec:graded_prime}).
\end{proof}

Let $\RGm : \Db(\Mod R) \to \Db(\Mod R)$ be the right derived functor of 
$\G := \colimit_n \Hom(R/\m^n,-)$, where $\Db(\Mod R)$ is the bounded derived category of $\Mod R$.
Recall that $H^i(\RGm(M)) = H^i_\m(M)$ for all $i$ and $M \in \Mod R$.
The usual spectral sequence argument of double complexes tells us that $\cpx L_R$ 
is a flat resolution of $\RGm (R)$, and therefore we have the following. 

\begin{cor}\label{derived tensor}
For a bounded complex $\cpx M$ of $R$-modules, $\RGm(\cpx M)$ and $\cpx L_R \tns_R \cpx M$ are isomorphic  
in $\Db(\Mod R)$.
\end{cor}

When $M$ is $\zM$-graded, by Lemma \ref{sec:localization}, $T^{-1}_\s R \tns_R M$ is also $\zM$-graded,
and moreover the differentials of $\cpx L_R \tns_R M$ are in $\Lgr R$.
Thus if $M \in \Lgr R$, $H^i(\cpx L_R \tns_R M)$ has a $\zM$-grading induced by $\cpx L_R \tns M$. 
Hence we have the following. 

\begin{cor}
$H^i_\m(M) \in \Lgr R$ for $M \in \Lgr R$.
\end{cor}

\section{Squarefree Modules}

In this section, we assume that all the toric face rings are cone-wise normal.
Let $(\C,\cell)$ be a conical complex, $\MM$ a monoidal complex, and $R$ the toric face ring of $\MM$.
For $a \in \sM$, there exists a unique cell
$\s \in \cell$ such that $\relint (C_\s) \ni a$. We denote this $\s$ by $\supp(a)$.

\begin{dfn}
An $R$-module $M \in \lgr R$ is said to be \term{squarefree} if it is $\MM$-graded and
the multiplication map $M_a \ni  x \mapsto t^b x \in M_{a+b}$  is an isomorphism 
of $\kk$-vector spaces for all $a,b\in \sM$ with $\supp(a+b) = \supp(a)$.
\end{dfn}

For a monomial ideal $I$ of $R$,  it is a squarefree $R$-module, if and only if  so is $R/I$, 
if and only if $I = \sqrt{I}$. In particular, $\p_\s$ and $R/\p_\s$ are squarefree.     
We denote, by $\Sq R$, the full subcategory of $\lgr R$ consisting of squarefree $R$-modules.
As in the case of affine semigroup rings or Stanley-Reisner rings (see \cite{Y01,Y02}), 
$\Sq R$ has nice properties. 
Since their proofs are also quite similar to these cases, we omit some of them.

\begin{lem}[{cf. \cite{Y01,Y02}}]\label{sec:maps}
Let $M \in \Sq R$. 
Then for $a,b \in \sM$ with $\supp(a) \ge \supp(b)$, there exists a $\kk$-linear map
$\varphi^M_{a,b}:M_b \to M_a$ satisfying the following properties:
\begin{enumerate}
\item $\varphi^M_{a,b}$ is bijective if $\supp(a) = \supp(b)$;
\item $\varphi^M_{a,a} = \idmap$ and $\varphi^M_{a,b}\circ \varphi^M_{b,c} = \varphi^M_{a,c}$
for $a,b,c \in \sM$ with $\supp(c) \le \supp(b) \le \supp(a)$;
\item For $a,a',b,b' \in \sM$ with $\supp(a) \le \supp(a')$ and $\supp(a+b) \le \supp(a'+b')$,
the following diagram
$$
\xymatrix{
M_a \ar[r]^{t^b} \ar[d]_{\varphi^M_{a',a}} & M_{a+b} \ar[d]^{\varphi^M_{a'+b',a+b}} \\
M_{a'} \ar[r]_{t^{b'}} & M_{a'+b'}
}
$$
commutes.
\end{enumerate}
\end{lem}

Let $\Lambda$ denote the incidence algebra of the regular cell complex $\cell$ over $\kk$
(regarding $\cell$ as a poset by its order $>$).
That is, $\Lambda$ is a finite dimensional associative $\kk$-algebra with basis
$\set{e_{\s,\t}}{\s, \t \in \cell \text{ with } \s \ge \t}$,
and its multiplication is defined by
$$
e_{\s,\t}\cdot e_{\t',\u} =\begin{cases}
e_{\s,\u} & \text{if $ \t = \t' $;} \\
0 & \text{otherwise.}
\end{cases}
$$

We write $e_\s := e_{\s,\s}$ for $\s \in \cell$. 
Each $e_\s$ is idempotent, and moreover $\Lambda e_\s$ is indecomposable as a left $\Lambda$-module.
It is easy to verify that $e_\s \cdot e_\t = 0$  if $\s \not= \t$
and that $1 = \sum_{\s \in \cell} e_\s$. 
Hence $\Lambda$, as a left $\Lambda$-module, can be decomposed as
$\Lambda = \Dsum_{\s \in \cell} \Lambda e_\s$.

Let $\mod \Lambda$ denote the category of finitely generated left $\Lambda$-modules. 
As a $\kk$-vector space, any $M \in \mod \Lambda$ has the decomposition  $M = \Dsum_{\s \in \cell} e_\s M$.  
Henceforth we set $M_\s := e_\s M$.

For each $\s \in \cell$,  we can construct an indecomposable injective object in 
$\mod \Lambda$ as follows;
set
$$
\bar E(\s) := \Dsum_{\t \in \cell, \, \t \le \s}\kk \, \bar e_\t,
$$
where $\bar e_{\t}$'s are basis elements. 
The multiplication on $\bar E(\s)$ from the left defined by
$$
e_{\u, \, \omega} \cdot \bar e_\t = \begin{cases}
\bar e_\u & \text{if $ \t = \omega $ and $\u \le \s $;} \\
0 & \text{otherwise,}
\end{cases}
$$
bring $\bar E(\s)$ a left $\Lambda$-module structure.
The following is well known.

\begin{prop}\label{mod Lambda} 
The category $\mod \Lambda$ is abelian and enough injectives, and any indecomposable injective 
object is isomorphic to $\bar E(\s)$ for some $\s \in \cell$.
\end{prop}

As in the case of affine semigroup rings and Stanley-Reisner rings, we have

\begin{prop}[{cf. \cite{Y01,Y02}}]\label{sec:Sq_cat}
There is an equivalence between $\Sq R$ and $\mod \Lambda$.
Hence $\Sq R$ is abelian, and enough injectives. Any indecomposable injective object in $\Sq R$
is isomorphic to $R/\p_\s$ for some $\s \in \cell$.
\end{prop}
\begin{proof}
First, we will show the category equivalence. 
The object $M \in \Sq R$ corresponding to $N \in \mod \Lambda$ is given as follows. 
Set $M_a := N_{\supp(a)}$ for each $a \in \sM$.  For $a, b \in \sM$ such that $a+b$ exists,  
define the multiplication 
$M_a \ni x \mapsto t^b\cdot x \in M_{a + b}$ by 
$$M_a = N_{\supp(a)} \ni x \longmapsto e_{\supp(a+b), \, \supp(a)} \cdot x \in N_{\supp(a+b)} = M_{a+b}.$$
Then $M$ becomes a squarefree module. See \cite{Y01,Y02} for details 
(though \term{right} $\Lambda$-modules are treated in \cite{Y01,Y02},
there is no essential difference).

Since $R/\p_\s$ corresponds to $\bar{E}(\s)$ in this equivalence, 
the other statements follow from Proposition~\ref{mod Lambda}.  
\end{proof}

Let $\Db(\Sq R)$ be the bounded derived category of $\Sq R$.
We shall define the functor $\DD: \Db(\Sq R) \to \Db(\Sq R)^\op$.
This functor will play an important role in the next section.
First, we choose elements $\bc(\s) \in \sM$ with $\supp(\bc(\s)) = \s$ for each $\s \in \cell$,
and set $\varphi^M_{\s, \t} := \varphi^M_{\bc(\s), \, \bc(\t)}$ for $M \in \Sq R$ and 
$\s, \t \in \cell$ with $\t \le \s$,  
where $\varphi_{\bc(\s), \, \bc(\t)}^M$ is the map given in Lemma \ref{sec:maps}.
To a bounded complex $\cpx M$ of squarefree $R$-modules, 
we assign the complex $\DD(\cpx M)$ defined as follows:
the component of cohomological degree $p$ is
$$
\DD(\cpx M)^p := \Dsum_{i + \dim C_\s = -p} (M^i_{\bc(\s)})^* \tns_\kk R/\p_\s,
$$
where $(-)^*$ denotes the $\kk$-dual, but the ``degree" of $(M^i_{\bc(\s)})^*$ is $0 \in \szM$.
Define $d':\DD(\cpx M)^p \to \DD(\cpx M)^{p+1}$ and $d'':\DD(\cpx M)^p \to \DD(\cpx M)^{p+1}$ by
$$
d'(y \tns r) = \sum_{\substack{\t \le \s,\\ \dim \t= \dim \s - 1}} 
\e(\s,\t) \cdot (\varphi^{M^i}_{\s,\t})^*(y) \tns \nat(r),
\quad d''(y\tns r) = (-1)^p \cdot (\partial^i_{\cpx M})^*(y) \tns r
$$
for $y \in M^i_{\bc(\s)}$ with $i + \dim C_\s = -p$ and $r \in R/\p_\s$. 
Here $\e(\s,\t)$ is an incidence function on $\cell$ and  
$\nat : R/\p_\s \to R/\p_\t$ 
is the natural surjection (note that $\p_\s \subset \p_\t$ if $\t \le \s$).
Clearly, $(\DD(\cpx M), d' + d'')$ forms a bounded complex in $\Sq R$,
and Lemma \ref{sec:maps} guarantees the independence of $\DD(\cpx M)$ 
from the choice of $\bc(\s)$'s.

Let $\Cb(\Sq R)$ be the bounded homotopy category of $\Sq R$. 
Since the  above assignment preserves mapping cones, it 
gives a triangulated functor of $\Cb(\Sq R) \to \Cb(\Sq R)^\op$, 
and an usual argument using spectral sequences indicates that it 
preserves quasi-isomorphisms. Hence it induces
the functor $\Db(\Sq R) \to \Db(\Sq R)^\op$, which is denoted by $\DD$ again.

Up to translation, the functor $\DD$ coincides with the functor 
$\mathbf D: \Db(\mod \Lambda) \to \Db(\mod \Lambda)^\op$ defined in \cite{Y05}, 
through the equivalence $\Sq R \cong \mod \Lambda$ in Proposition \ref{sec:Sq_cat}.
Hence by \cite[Theorem 3.4 (1)]{Y05}, we have the following.

\begin{prop}
The functor $\DD: \Db(\Sq R) \to \Db(\Sq)^\op$ satisfies $\DD \circ \DD \cong \idmap$.  
\end{prop}

\section{Dualizing complexes}

We first recall the following useful result due to Sharp (\cite{Sh}).
\begin{thm}[Sharp]\label{sec:sharp}
Let $A$ and $B$ be commutative noetherian rings, and $f: A \to B$ a ring homomorphism.
Assume that $A$ has a dualizing complex $\cpx D_A$ and $B$, regarded as an $A$-module by $f$, 
is finitely generated. Then $\Hom_A(B, \cpx D_A)$ is a dualizing complex of $B$.
\end{thm}

For a commutative ring $A$, we denote, by $E_A(-)$, the injective hull in $\Mod A$.
Let $(\C,\cell)$ be a conical complex, $\MM$ a cone-wise normal monoidal complex supported by $\C$,
and $R := \fring \MM$ its toric face ring. Since $R$ is a finitely generated $\kk$-algebra,
we can take a polynomial ring which surjects onto $R$.
Thus, Proposition \ref{sec:sharp} implies that $R$ has a normalized dualizing complex
$$
\cpx D_R: 0 \to \Dsum_{\substack{\p \in \Spec R,\\ \dim R/\p = d}} E_R(R/\p) \to
\Dsum_{\substack{\p \in \Spec R,\\ \dim R/\p = d-1}} E_R(R/\p) \to \cdots \to
\Dsum_{\substack{\p \in \Spec R,\\ \dim R/\p = 0}} E_R(R/\p) \to 0,
$$
where $d := \dim R = \dim \cell + 1$ and cohomological degrees are given by
$$
D^i_R := \Dsum_{\substack{\p \in \Spec R,\\ \dim R/\p = -i}}\E_R(R/\p).
$$

On the other hand, set $$I^{i}_R := 
\Dsum_{\substack{\s \in \cell \\ \dim R/\p_\s = -i}}R/\p_\s$$
for $i = 0, \dots , d$, and define $I^{-i}_R \to I^{-i+1}_R$
by
$$
x \mapsto \sum_{\substack{\dim \fring{\t} =i-1 \\ \t \leq \s}}\e(\s,\t) \cdot \nat(x)
$$
for $x \in R/\p_\s \subset I^{-i}_R$,
where $\e(\s,\t)$ denotes an incidence function of $\cell$, 
and $\nat$ is the natural surjection $R/\p_\s \to R/\p_\t$. Then
$$
\cpx I_R: 0 \longto I^{-d}_R \longto I^{-d + 1}_R \longto \cdots \longto I^0_R \longto 0
$$
is a complex. 

\begin{thm}\label{sec:ishida}
With the above situation (in particular, $R$ is cone-wise normal), 
$\cpx I_R$ is quasi-isomorphic to the normalized dualizing complex $\cpx \D_R$ of $R$.
\end{thm}

For the embedded case, Theorem \ref{sec:ishida} was already shown by Ichim and R\"omer \cite{IR}, 
using the natural $\ZZ^n$-graded structure.
However, in the general case, we cannot apply the same argument.

\begin{prop}\label{sec:subcpx}
With the hypothesis in Theorem \ref{sec:ishida}, $\cpx I_R$ is a subcomplex of $\cpx \D_R$.
\end{prop}
\begin{proof}
We shall go through some steps.

\medskip

\noindent{\it Step} 1. Some observations.

\medskip

For $\s \in \cell$, we set $\fring \s := R/\p_\s \cong \fring{\M_\s}$ and 
$d_\s := \dim C_\s = \dim \fring\s = \dim \s +1$.  
Note that $$\cpx D_\s := \Hom_R(\fring\s, \cpx D_R)$$ is a normalized dualizing complex of $\fring\s$
by Proposition \ref{sec:sharp}. Since $\fring\s$ is $\ZZ^{d_\s}$-graded, 
we also have the $\ZZ^{d_\s}$-graded version of a normalized dualizing complex
$$
\cpx{\D*_\s}: 0 \to \Dsum_{\substack{\t \leq \s,\\ \dim \fring \t = d_\s}} \E*_{\fring\s}(\fring\t)
\to \Dsum_{\substack{\t \leq \s,\\ \dim \fring \t = d_\s - 1}} \E*_{\fring\s}(\fring\t)
\to \cdots \to \E*_{\fring\s}(\kk) \to 0,
$$
where $\E*_{\fring\s}(-)$ denotes the injective hull in the category of 
$\ZZ^{d_\s}$-graded ${\fring\s}$-modules,  and
cohomological degrees are given by the same way as $\cpx D_R$. 

It is easy to see that the \term{positive part} 
$$
\bigoplus_{a \in \M_\s} [{\E*}_{\fring \s}(\fring \t)]_a
$$ 
of ${\E*}_{\fring \s}(\fring \t)$ 
is isomorphic to $\kk[\t]$.  Set 
\begin{equation}\label{eq:positive}
\cpx I_\s := \bigoplus_{a \in \M_\s} [\cpx {\D*_\s}]_a \subset \cpx {\D*_\s}.
\end{equation}
Clearly, $\cpx I_\s$ is a complex with  
\begin{equation}\label{eq:positive2}
I^i_\s: = \Dsum_{\substack{\t \le \s,\\ \dim \fring{\t} = -i}} \fring\t. 
\end{equation}

As is well-known, $\cpx{D_\s}$ is an injective resolution of $\cpx {\D*_\s}$ in the category 
$\Mod {\fring \s}$, and the latter can be seen as a subcomplex of  the former  
in a non-canonical way. By the construction,  $\cpx I_\s$ is a subcomplex of  $\cpx{\D*_\s}$, and 
$\cpx D_\s$ is a subcomplex of $\cpx D_R$. 
Combining them, we have an embedding  $\cpx I_\s \embto \cpx D_R$. 
Thus the problem is the compatibility of the embeddings $\cpx I_\s \embto \cpx D_R$ and $\cpx I_\t \embto \cpx D_R$ for $\s,\t \in \C$.

\medskip

\noindent{\it Step} 2. Canonical embedding $\fring \s \embto \cpx \D_R$.

\medskip

For $\s \in \cell$, let $\W_{\fring\s}$ be the canonical module of $\fring\s$.
By our hypothesis that $\MM$ is cone-wise normal, we see that $\W_{\fring\s}$ is just 
the  ideal generated by $\set{t^a \in \fring \s}{a \in \relint (C_\s)\cap \M_\s}$ 
(cf. \cite[Theorem 6.3.5]{BH}).
Whence we have the exact sequence:
$$
0 \longto \W_{\fring \s} \longto \fring \s \longto \fring\s /\W_{\fring\s} \longto 0.
$$
Since $\Hom_R(\fring\s /\W_{\fring\s},E_R(\fring \s)) = 0$,
applying $\Hom_R(-,E_R(\fring\s))$ to the above exact sequence yields the canonical isomorphism
$$
\Hom_R(\fring \s,E_R(\fring\s)) \cong \Hom_R(\W_{\fring \s}, E_R(\fring\s)),
$$
and thus the canonical embedding
\begin{align}\label{eq:inc1}
\Hom_R(\W_{\fring \s}, E_R(\fring\s)) \cong \set{x \in \E_R(\fring\s)}{\p_\s x = 0} \subset \E_R(\fring\s).
\end{align}
Since we have 
\begin{align*}
\Hom_R(\W_{\fring\s}, \D^{-d_\s}_R) = \Dsum_{\substack{\p \in \Spec R,\\ \dim R/\p = d_\s}}\Hom_R(\W_{\fring\s}, \E_R(R/\p))
                                 = \Hom_R(\W_{\fring\s}, \E_R(\fring\s)),
\end{align*}
in conjunction with \eqref{eq:inc1}, we obtain the canonical embedding
$$
\Hom_R(\W_{\fring \s}, \D^{-d_\s}_R) \subset \E_R(\fring\s) \subset \D^{-d_\s}_R.
$$

Since $\Hom_R(\W_{\fring\s}, \D^{-d_\s-1}_R) = 0$, it follows that
\begin{eqnarray*}
\Ext^{-d_\s}_R(\W_{\fring\s}, \cpx D_R) 
&=& \Ker( \, \Hom_R(\W_{\fring\s}, \D^{-d_\s}_R) \to \Hom_R(\W_{\fring\s}, D^{-d_\s+1}_R) \, )\\
&=& \{ \, x \in D_R^{-d_\s} \mid  \text{$\p_\s x =0$ and $\partial(J_\s x)=0$ } \}, 
\end{eqnarray*}
where $J_\s := \{ \,  t^a  \mid a \in \relint (C_\s)\cap \M_\s \, \}$ and 
$\partial: D^{-d_\s} \to D^{-d_\s +1}$ is the differential map.  
Consequently, we have
\begin{align}\label{eq:emb1}
\fring\s \cong \Ext^{-d_\s}_R(\W_{\fring\s}, \cpx D_R) \subset \D^{-d_\s}_R
\end{align}
canonically.

Using this, we have a canonical injection
\begin{equation}\label{eq:emb1.5} 
I_R^i = \Dsum_{\substack{\s \in \cell \\ \dim \fring{\s} = -i}} \fring\s \embto D_R^i
\end{equation}
for each $i$. 
\medskip

\noindent{\it Step} 3. Compatibility.

\medskip

For $\s,\t \in \cell$ with $\t \leq \s$, set 
$$\uExt_{\fring\s}^i(\W_{\fring\t},  \cpx{\D*}_\s) := 
H^i( \, \Hom^\bullet_{\fring \s}(\W_{\fring\t}, \cpx{\D*}_\s) \,).$$  
This module has a $\ZZ^{d_\s}$-grading, since so does  $\W_{\fring\t}$.  
Applying the same argument as in Step 2 
(replacing $R$ by $\fring\s$ and $\cpx D_R$ by $\cpx{\D*}_\s$), we have a canonical embedding  
which is the first injection of the sequence 
\begin{align}\label{eq:emb2}
\fring\t \cong \uExt^{-d_\t}_{\fring\s}(\W_{\fring\t}, \cpx{\D*}_\s) 
\embto \D*^{-d_\t}_\s \embto \D^{-d_\t}_R. 
\end{align}
Here the last injection is not canonical. 
Since the inclusions $\cpx {\D*}_\s \embto \cpx D_\s \embto \cpx D_R$  
give the isomorphisms 
$$\uExt_{\fring\s}^{-d_\t}(\W_{\fring\t}, \cpx{\D*}_\s) 
\cong \Ext_{\fring \s}^{-d_\t}(\W_{\fring\t}, \cpx D_\s)
\cong \Ext_R^{-d_\t}(\W_{\fring\t}, \cpx D_R),$$ 
the embedding $\fring\t \embto D_R^{-d_\t}$ given in \eqref{eq:emb2}
coincides with the one given in Step~2.  
(So the image of  \eqref{eq:emb2} does not depend on 
the choice of an injection $\D*^{-d_\t}_\s \embto \D^{-d_\t}_R$.)

It is easy to see that the inclusion \eqref{eq:positive} (see also \eqref{eq:positive2}) 
is same as the one given by \eqref{eq:emb2}. 
Therefore, through any $\cpx{\D*_\s} \embto \cpx D_R$,  
the embeddings of \eqref{eq:positive} and \eqref{eq:emb1.5} are compatible. 
So under this embedding, we have $I_\s^i \subset I_R^i \subset D_R^i$.  
Since $\cpx I_\s$ is a subcomplex of $\cpx D_R$ for all $\s \in \cell$, 
$\bigoplus_{i \in \ZZ} I_R^i$ forms a subcomplex of $\cpx D_R$. 

\medskip

We can take a generator $1_\s \in \fring \s 
\subset I_R^{-d_\s} \subset  D_R^{-d_\s}$ for each $\s \in \cell$ satisfying
$$\partial_{\cpx D_R}(1_\s) = \sum \e  '(\s, \t) \cdot 1_\t$$ 
for some incidence function $\e '$ on $\cell$. 
Recall that we have fixed an incidence function $\e$ to define the differential of $\cpx I_R$. 
While $\e$ and $\e '$ do not coincide in general, their difference is well-regulated  
(cf. \cite[p.265]{BH}).  So, after a suitable change of  $\{1_\s \}_{\s \in \cell}$, 
we have 
$$\partial_{\cpx D_R}(1_\s) = \sum \e  (\s, \t) \cdot 1_\t.$$ 
Therefore we conclude that $\cpx I_R$ is a subcomplex of $\cpx D_R$ as is desired.
\end{proof}

When $R$ is a normal semigroup ring, the second author showed in \cite[Lemma~3.8]{Y06} that there is a natural isomorphism between
$\DD$ and $\RHom(-,\cpx \D_R)$. The next result generalizes this to toric face rings.

\begin{prop}\label{sec:DD}
There is the following commutative diagram;
$$
\xymatrix{
\Db(\Sq R) \ar[r]^{\for} \ar[d]_{\DD} & \Db(\Mod R) \ar[d]^{\RHom(-,\cpx \D_R)} \\
\Db(\Sq R)^{\op} \ar[r]_{\for} & \Db(\Mod R)^{\op},
}
$$
where $\for$ is the functor induced by the forgetful functor $\Sq R \to \Mod R$.
In particular, we have $\DD(\cpx M) \cong \RHom_R(\cpx M, \cpx D_R)$ in 
$\Db(\Mod R)$ for any $\cpx M \in \Db(\Sq R)$, and hence 
$\Ext^i_R(\cpx M,\cpx D_R)$ has a $\zM$-grading induced by $\DD(\cpx M)$.
\end{prop}

\begin{proof}
Let $\InjSq$ be the full subcategory of $\Sq R$ consisting of all injective objects, that is, 
finite direct sums of $\kk[\s]$ for various $\s \in \cell$. 
As is well-known (cf. \cite[Proposition~4.7]{Ha}), the bounded homotopy category $\Cb(\InjSq)$ is equivalent to 
$\Db(\Sq R)$.  
It is easy to see that $\DD(\kk[\s]) = \Hom^\bullet_R(\kk[\s], \cpx I_R)$. 
Moreover, $\DD(\cpx J) = \Hom^\bullet_R(\cpx J, \cpx I_R)$ for all $\cpx J \in 
\Cb(\InjSq)$. Since $\cpx I_R$ is a subcomplex of $\cpx D_R$ as shown in  
Proposition~\ref{sec:subcpx}, we have a chain map 
$\Hom^\bullet_R(\cpx J, \cpx I_R) \to  \Hom^\bullet_R(\cpx J, \cpx D_R)$. 
This map induces a natural transformation $\Psi: \for \circ \DD \to \RHom_R(-, \cpx D_R) \circ \for$. 
If $M \in \Sq R$ is a $\kk[\s]$-module, then 
$\DD(M) \cong \RHom_{\kk[\s]}(M, \cpx D_\s) \cong \RHom_R(M, \cpx D_R)$ 
by \cite[Lemma~3.8]{Y06}.  In particular, 
$\Psi(\kk[\s])$ is isomorphism for all $\sigma \in \cell$.  
Hence applying \cite[Proposition 7.1]{Ha}, we see that $\Psi(\cpx M)$ 
is an isomorphism for all $\cpx M \in \Db(\Sq R)$.
\end{proof}

The most part of the proof of Theorem \ref{sec:ishida} has done now.

\begingroup
\renewcommand{\proofname}{Proof of Theorem \ref{sec:ishida}}
\begin{proof}
Since $R \in \Sq R$, we have
$
I_R = \DD(R) \cong \RHom(R, \cpx D_R) \cong \cpx D_R
$
by Proposition \ref{sec:DD}.
\end{proof}
\endgroup

Let $M \in \Lgr R$.  
We can construct the {\it graded Matlis dual} $M^\vee \in \Lgr R$ of $M$ as follows: 
For each $a \in \szM$, $(M^\vee)_a$ is the $\kk$-dual space of $M_{-a}$.  
%(since $a \in \ZZ \M_\s$ for some $\s \in \Sigma$, $-a \in \ZZ \M_\s$ gives $-a \in \szM$). 
For $a,b \in \szM$ such that  $a+b$ exists (that is, $a,b,a+b \in \M_\s$ for some 
$\s \in \cell$), the multiplication map $(M^\vee)_a \ni x \longmapsto  t^b x \in (M^\vee)_{a+b}$ is the $\kk$-dual of 
$M_{-a-b} \ni y \longmapsto  t^b y \in M_{-a}$. Otherwise, $t^bx=0$ for all $x \in (M^\vee)_a$.  

It is obvious that $M^\vee$ is actually a $\ZZ\MM$-graded $R$-module. 
If $\dim_\kk M_a < \infty$ for all $a \in \szM$ (e.g. $M \in \lgr R$), then $M^{\vee\vee} \cong M$. 
Clearly, $(-)^\vee$ defines an exact contravariant functor from $\Lgr R$ to itself. 
We can extend this functor to the functors $\Cb(\Lgr R) \to \Cb(\Lgr R)^\op$ and 
$\Db(\Lgr R) \to \Db(\Lgr R)^\op$. We simply denote them by $(-)^\vee$. 

\begin{prop}\label{Local Duality}
As functors from $\Db(\Sq R)$ to $\Db(\Lgr R)$, we have 
$\RGm \cong (-)^\vee \circ \for \circ \DD$,
where $\for :\Db(\Sq R) \to \Db(\Lgr R)$ is induced by the forgetful functor $\Sq R \to \Lgr R$.
In particular, if $M \in \Sq R$, then 
$H_\m^i(M) \cong \Ext^{-i}_R(M, \cpx D_R)^\vee$ as $\ZZ\MM$-graded modules for all $i$.  
\end{prop}

\begin{proof}
We use the notation of the proofs of the above results.  
If $M \in \Lgr R$, then the {\it $\sM$-graded part} $\bigoplus_{a \in \sM} M_a$ of $M$ is clearly 
an $R$-submodule. For $\tau \in \Sigma$, recall that $T_\tau = \{ \, t^a \mid a \in \M_\tau \, \}$ 
is a multiplicatively closed set. 
It is easy to see that, for $\s, \tau \in \Sigma$, the localization $T_\tau^{-1} \kk[\s]$ 
is non-zero if and only if $\tau \leq \s$.  When $\tau \leq \s$, the $\sM$-graded part of 
$(T_\tau^{-1} \kk[\s])^\vee$ is isomorphic to $\kk[\tau]$.  

Let $\cpx L_R$ be the C\v ech complex of $R$ defined in Section 3. 
It is easy to see that  the $\sM$-graded part of $(\cpx L_R \tns_R \kk[\s])^\vee$ is isomorphic to $\DD(\kk[\s])$. 
Moreover, if $\cpx J \in \Cb(\InjSq)$, then the $\sM$-graded part of $(\cpx L_R \tns_R \cpx J)^\vee$ 
is isomorphic to $\DD(\cpx J)$. Thus  $\DD(\cpx J)$ is a subcomplex of $(\cpx L_R \tns_R \cpx J)^\vee$, 
and there is a chain map $\cpx L_R \tns_R \cpx J \to \DD(\cpx J)^\vee$. 
Recall that $\cpx L_R \tns_R \cpx J$ is quasi-isomorphic 
to $\RGm (\cpx J)$ by Corollary~\ref{derived tensor}. 
Hence we have a natural transformation $\Phi: \RGm \to  (-)^\vee \circ \for \circ \DD$, 
where we regard $\RGm$ and $(-)^\vee \circ \for \circ \DD$ as functors from 
$\Cb(\InjSq)$ $(\cong \Db(\Sq R))$ to $\Db(\Lgr R)$. 
Since $\Phi(\kk[\s])$ is an isomorphism for all $\s \in \cell$, $\Phi$ is a natural isomorphism by
\cite[Proposition 7.1]{Ha}. 
\end{proof}

\section{Sheaves associated with squarefree modules}
Throughout this section, $\MM$ is a cone-wise normal monoidal complex supported by 
a conical complex $(\C, \cell)$.  
Recall that $X= \bigcup_{\sigma \in \cell} \sigma$ is the 
underlying topological space of the cell complex $\cell$.  As in the previous section, 
let $\Lambda$ be the incidence algebra of the poset $\cell$ over $\kk$, 
and $\mod \Lambda$ the category of finitely generated left $\Lambda$-modules.  

Let $\Sh(X)$ be the category of sheaves of finite dimensional 
$\kk$-vector spaces on $X$. We say $\cF \in \Sh(X)$ is {\it constructible} 
with respect to the cell decomposition $\cell$, if the restriction 
$\cF|_\sigma$ is a constant sheaf for all $\emptyset \ne \sigma \in \cell$. 
%Here $\cF|\sigma$ denotes the inverse image $i^* \cF$ of $\cF$ 
%by the embedding map $i: \sigma \to X$.  

In \cite{Y05}, the second author constructed 
the functor $(-)^\dagger : \mdL \to \Sh(X)$. 
(Under the convention that $\emptyset \not \in \cell$, this functor has been  well-known to specialists.) 
Here we give a precise construction for the reader's convenience.   

For $M \in \mdL$, set $$\sp(M) 
:= \bigcup_{\emptyset \ne \sigma \in \cell } \sigma \times M_\sigma.$$
Let $\pi : \sp(M) \to X$ be the projection map which sends $(p, m) \in 
\sigma \times M_\sigma \subset \sp(M)$ to $p \in \sigma \subset X$. 
For an open subset $U \subset X$ and a map $s: U \to \sp(M)$, 
we will consider the following conditions:  

\begin{itemize}
\item[$(*)$]  $\pi \circ s = \idmap_{U}$ and $s_p = e_{\s, \, \t} 
\cdot s_q$ for all $p \in \sigma \cap U$, $q \in \tau \cap U$ with $\s\geq \t$. 
Here $s_p$ (resp. $s_q$) is the element of $M_\s$ 
(resp. $M_\t$) with $s(p) = (p, s_p)$ (resp. $s(q) = (q, s_q)$).  
\item[$(**)$] There is an open covering $U = \bigcup_{i \in I} 
U_i$ such that the restriction of $s$ to $U_i$ satisfies $(*)$ for all $i \in I$. 
\end{itemize}

Now we define a sheaf $M^\dagger \in \Sh(X)$ from $M$ as follows. 
For an open set $U \subset X$, set 
$$M^\dagger(U):= 
\{ \, s \mid \text{$s: U \to \sp(M)$ is a map satisfying $(**)$} \,\}$$
and the restriction map $M^\dagger(U) \to M^\dagger(V)$ is the natural one. 
It is easy to see that $M^\dagger$ is a constructible sheaf with respect to the cell decomposition $\cell$. 
For $\sigma \in \cell$, let $U_\sigma := \bigcup_{\tau \geq \sigma} \tau$ 
be an open set of $X$. Then we have $M^\dagger(U_\sigma) \cong M_\sigma$. 
Moreover, if $\sigma \leq \tau$, then we have $U_\sigma \supset U_\tau$ and 
the restriction map $M^\dagger(U_\sigma) \to M^\dagger(U_\tau)$ 
corresponds to the multiplication map $M_\sigma \ni x \mapsto 
e_{\tau, \, \sigma} x \in M_\tau$. 
For a point $p \in \sigma$, the stalk $(M^\dagger)_p$ of $M^\dagger$ at 
$p$ is isomorphic to $M_\sigma$. 
This construction gives the exact functor $(-)^\dagger:\mdL \to \Sh(X)$. 
We also remark that 
$M_\emptyset$ is irrelevant to $M^\dagger$.

As in the previous sections, let $R= \kk[\MM]$ be the toric face ring, 
and $\Sq R$ the category of squarefree $R$-modules. 
Through the equivalence $\Sq R \cong \mdL$, $(-)^\dagger: \mdL \to \Sh(X)$ gives  
the exact functor $$(-)^+: \Sq R \to \Sh(X).$$ 

Recall that  $X$ admits Verdier's dualizing complex 
$\cDx \in \Db(\Sh(X))$ with coefficients in $\kk$ (see \cite[V. Section 2]{Iver}). 
In \cite{Y05}, the second author considered the duality functor 
$\DDD : \Db(\mdL) \to \Db(\mdL)$. 
Through the functor $(-)^\dagger: \mdL \to \Sh(X)$, $\DDD$ corresponds to 
Poincar\'e-Verdier duality on $\Db(\Sh(X))$. More precisely, 
\cite[Theorem~3.2]{Y05} states that, for $\cpx M \in \Db(\mdL)$, we have 
$$\DDD(\cpx M)^\dagger \cong \RcHom((\cpx M)^\dagger, \cDx)$$
in $\Db(\Sh(X))$. 
On the other hand, through the equivalence $\mdL \cong \Sq R$, 
the duality $\DDD$ on $\Db (\mdL)$ corresponds to our duality $\DD$ on $\Db(\Sq R)$ 
up to translation. More precisely, $\DD(-)[-1]$ corresponds to $\DDD(-)$, 
where the complex $\cpx M[-1]$ of a complex $\cpx M$
denotes the degree shifting of $\cpx M$ with $\cpx M[-1]^i = M^{i-1}$. So we have the following.

\begin{thm}\label{Verdier}
For $\cpx M \in \Db(\Sq R)$, we have 
$$\DD(\cpx M)^+ [-1] \cong \RcHom((\cpx M)^+, \cDx)$$
in $\Db(\Sh(X))$.  In particular, $(\cpx I_R)^+[-1] \cong \cDx$, where 
$\cpx I_R$ is the complex constructed in the previous section. 
\end{thm}

By Proposition~\ref{Local Duality},  if $M \in \Sq R$, 
then we have $H_\m^i(M)^\vee \cong \Ext_R^{-i}(M, \cpx D_R) \in \Sq R$. 
Hence $H_\m^i(M)$ is $-\sM$-graded and 
the next result determines the ``Hilbert function" of $H_\m^i(M)$.  

\begin{thm}\label{local cohomology}
If $M \in \Sq R$, we have the following.  
\begin{itemize}
\item[(a)] There is an isomorphism  
$$H^i(X, M^+) \cong [H_\m^{i+1}(M)]_0 \quad  \text{for all $i \geq 1$},$$ 
and an exact sequence 
$$0 \to [H_\m^0(M)]_0 \to M_0 \to H^0( X, M^+) \to [H_\m^1(M)]_0 \to 0.$$
\item[(b)] If $0 \ne a \in \sM$ with $\sigma = \supp(a)$, then 
$$[H_\m^i(M)]_{-a} \cong H^{i-1}_c(U_\s, M^+|_{U_\s})$$
for all $i \geq 0$. Here $U_\s = \bigcup\limits_{\tau \geq \s} \tau$ is an open set of $X$, 
and $H^\bullet_c(-)$ stands for the cohomology with compact support. 
\end{itemize}
\end{thm}

\begin{proof}
(a) We have $H^i(\DD(M)) \cong \Ext_R^i(M, \cpx D_R) \cong H_\m^{-i}(M)^\vee$ by 
Proposition~\ref{Local Duality}. On the other hand, via the equivalence $\Sq R \cong \mdL$, 
 $\DD(-)[-1]$ corresponds to the duality $\DDD(-)=\RHom_\Lambda(-, \cpx \omega)$ 
of $\Db(\mdL)$ introduced in \cite{Y05}.  So the assertion follows from \cite[Corollary~3.5, Theorem~2.2]{Y05}. 

(b) Similarly, it follows from \cite[Lemma~5.1]{Y05}.
\end{proof}

In the sequel,  $\rH^i(X;\kk)$ denotes  the $i^{\rm th}$ reduced cohomology of 
$X$ with coefficients in $\kk$. That is, $\rH^i(X;\kk) \cong H^i(X;\kk)$ for all $i \geq 1$, and 
 $\rH^0(X;\kk) \oplus \kk \cong H^0(X;\kk)$. Here $H^i(X;\kk)$  is  the usual 
cohomology of $X$ with coefficients in $\kk$.

\begin{cor}[cf. Brun et al. {\cite[Theorem~1.3]{BBR}}]\label{reduced cohomology}
With the above notation, we have $[H_\m^i(R)]_0 \cong \rH^{i-1}(X; \kk)$ and 
$[H_\m^i(R)]_{-a} \cong \rH^{i-1}_c(U_\s, \const_{U_\s})$ for all $i \geq 0$
and all $0 \ne a \in \sM$. Here $\s = \supp (a)$, and 
$\const_{U_\s}$ is the $\kk$-constant sheaf on $U_\s$. 
\end{cor}

\begin{proof}
The second isomorphism is a direct consequence of Theorem~\ref{local cohomology} (b) 
and the fact that $R^+ \cong \const_X$. So it suffices to show the first. 
By the isomorphism of Theorem~\ref{local cohomology} (a),  
$[H_\m^i(R)]_0 \cong H^{i-1} (X, R^+) \cong H^{i-1}(X, \const_X) \cong H^{i-1}(X;\kk) \cong \rH^{i-1}(X;\kk)$ 
for all $i \geq 2$.  
Similarly, by the exact sequence of the theorem and that $H^0_\m(R)=0$, 
we have  $0 \to R_0 \to H^0( X; \kk) \to [H_\m^1(R)]_0 \to 0.$ 
Since $R_0 = \kk$, we have $[H_\m^1(R)]_0 \cong \rH^0(X;\kk).$
\end{proof}

We say $R$ is a {\it Buchsbaum ring}, if $R_{\m'}$ is a Buchsbaum local ring for all maximal ideal $\m'$.  
See \cite{SV} for further information.

\begin{thm}\label{Buchsbaum}
Set $\dim X= d$ (equivalently, $\dim R = d+1$). 
Then $R$ is Buchsbaum if and only if $\cH^i(\cDx) = 0$ for all $i \ne -d$. 
In particular, the Buchsbaum property of $R$ 
is a topological property of $X$ (while it might depend on $\chara (\kk)$). 
\end{thm}

\begin{proof}
Assume that $\cH^i(\cDx) \ne 0$ for some $i \ne -d$ (equivalently, $-d +1 \leq  i \leq 0$). 
Then $[H^{i-1}(\cpx I_R)]_a \ne 0$ for some $0 \ne a \in \sM$ by Theorem~\ref{Verdier}. 
Since $H^{i-1}(\cpx I_R)$ is squarefree, we have $\dim_\kk  ( H^{i-1}(\cpx I_R) \otimes_R R_\m )= \infty$. 
Since $H^{i-1}(\cpx I_R) \otimes_R R_\m$ is the Matlis dual of $H_\m^{1-i}(R_\m)$ 
over the local ring $R_\m$, we have $\dim_\kk H_\m^{1-i}(R_\m) = \infty$ and $R_\m$ is not Buchsbaum. 

Conversely, assume that $\cH^i(\cDx) = 0$ for all $i \ne -d$. 
Then  $H^i(\cpx I_R) = [H^i(\cpx I_R)]_0$ for all  $i \ne -d-1$, and 
they are $\kk$-vector spaces (that is, $R/\m$-modules). 
Hence  $H^i(\cpx I_R) \otimes_R R_{\m'} = 0$ for all $i \ne -d-1$ and all 
$\m'$  with $\m' \ne \m$. Thus $R_{\m'}$ is Cohen-Macaulay (in particular, Buchsbaum). 
It remains to show that $R_\m$ is Buchsbaum. 
Set $\cpx T := \tau_{-d-1} \cpx I_R$. Here, for a complex $\cpx M$ and an integer $r$, 
$\tau_{-r} \cpx M$ denotes the truncated complex 
$$\cdots \longrightarrow 0 \longrightarrow \operatorname{Im} (M^{-r} \to M^{-r+1}) 
\longrightarrow M^{-r+1} \longrightarrow M^{-r+2} \longrightarrow \cdots.$$ 
By the assumption, we have $H^i(\cpx T) = [H^i(\cpx T)]_0$ for all $i$. 
Since $\cpx T$ is a complex of $\MM$-graded modules, 
$\cpx U := \bigoplus_{0 \ne a \in \sM} (\cpx T)_a$  
is a subcomplex of $\cpx T$, and a natural map $\cpx T \to (\cpx T /\cpx U)$ 
is a quasi-isomorphism by the above observation.  
Since $\cpx T /\cpx U$ is a complex of $\kk$-vector spaces, 
$R_\m$ is Buchsbaum by \cite[II.Theorem~4.1]{SV}. 
\end{proof}

If $\dim X = d$ and $R$ is Buchsbaum, we set $or_X := \cH^{-d}(\cDx) \in \Sh(X)$. 
The next fact follows from \cite[IX, (4.1)]{Iver}. 

\begin{prop}[Poincar\'e duality]\label{Poincare} 
With the above situation, we have $H^i(X;\kk) \cong H^{d-i}(X, or_X)$ for all $i$. 
\end{prop}

If $X$ is a $d$-dimensional manifold (with or without boundary), then $R$ is Buchsbaum and 
$or_X$ is the usual {\it orientation sheaf}
of $X$ with coefficients in $\kk$ (see, for example, \cite[III, \S 8]{Iver}). 
When $X$ is an orientable manifold, then $or_X \cong \const_X$. 
In this case, Proposition~\ref{Poincare} is nothing other than the classical Poincar\'e duality.

Assume that  $\dim X = d$, equivalently, $\dim R = d+1$. 
If $R$ is Buchsbaum, we call $\omega_R := H^{-d-1}(\cpx I_R) \in \Sq R$ 
the {\it canonical module} of $R$. Clearly, $(\omega_R)^+ \cong or_X$.

\begin{exmp}
Recall the toric face ring $R$ given in Example~\ref{sec:Moebius_ring}, 
whose underlying topological space $X$ is the M\"obius strip. 
Clearly, $X$ is a manifold with boundary and $R$ is Buchsbaum. 
It is easy to see that $\rH^2(X;\kk)=0$ and 
$or_X \cong i_!\, \const_{X \setminus \partial X}$, 
where $\const_{X \setminus \partial X}$ is the $\kk$-constant sheaf on 
$X \setminus \partial X$ ($\partial X$ denotes the boundary of $X$), and 
$i: X \setminus \partial X \hookrightarrow X$ is the embedding map.  
Hence the canonical module $\omega_R$ is isomorphic to 
the monomial ideal $I$ with $I^+ \cong i_! 
\underline{\kk}_{X \setminus \partial X}$. 
So we have $\omega_R \cong (X_xX_u, X_zX_w, X_vX_y, X_xX_z, X_yX_w, X_xX_v)$, where 
the right side is an ideal of $R$. 
\end{exmp}
   
We say $R$ is {\it Gorenstein*}, if it is Cohen-Macaulay and $\omega_R \cong R$ 
as $\ZZ \MM$-graded modules. %In this case, $R$ is Gorenstein in the usual sense. 

\begin{thm}\label{CM & Gor} 
Set $d:=\dim X$. 
\begin{itemize}
\item[(a)] (Caijun, \cite{Cj}) 
$R$ is Cohen-Macaulay if and only if $\cH^i(\cDx) = 0$ for all $i \ne -d$, and  
$\rH^i(X;\kk)=0$ for all $i \ne d$. 
\item[(b)] Assume that $d \geq 1$ and $R$ is Cohen-Macaulay. Then 
$R$ is Gorenstein*, if and only if $or_X \cong \const_X$, if and only if  
$(or_X)_p \cong \kk$ for all $p \in X$ and $H^d(X;\kk) \ne0$. 
Here $\const_X$ denotes the $\kk$-constant sheaf on $X$ and 
$(or_X)_p$ is the stalk of the sheaf $or_X$ at $p$.  
\end{itemize}
\end{thm}

\begin{proof}
(a) Since $\dim R = d+1$, 
$R$ is Cohen-Macaulay if and only if $H^i(\cpx I_R) \, 
(= \Ext_R^i(R, \cpx D_R) )= 0$ for all $i \ne -d-1$.   
By Theorem~\ref{Verdier}, the above conditions are also equivalent to that 
$\cH^i(\cDx)=0$ for all $i \ne -d$ and $[H^i(\cpx I_R)]_0 = 0$ for all $i \ne -d-1$.  
%By Theorem~\ref{Local Duality} and Corollary~\ref{reduced cohomology}, 
Since  $[H^i(\cpx I_R)]_0 \cong ([H_\m^{-i}(R)]_0)^* \cong \rH^{-i-1}(X;\kk)^*$, we are done.  

(b)  We show the first equivalence. 
If $R$ is Gorenstein*, then $or_X \cong (\omega_R)^+ \cong R^+ \cong \const_X$. 
So we get the necessity. Next assume that $or_X \, (= (\omega_R)^+) \, \cong \const_X$. 
Then we have that 
\begin{equation}\label{canonical module}
[\omega_R]_a = \kk \qquad \text{for all $0 \ne a \in \sM$.}
\end{equation}
On the other hand, by Proposition~\ref{Poincare}, we have $[\omega_R]^\vee_0 \cong  
[H_\m^{d+1}(R)]_0 \cong H^d(X;\kk) \cong H^0(X, or_X) \cong H^0(X; \kk) \cong \kk$ 
(since $R$ is Cohen-Macaulay and $d \geq 1$, $\rH^0(X;\kk)=0$ and $X$ is connected). 
Take a non-zero element $x \in [\omega_R]_0$. 
Since $\omega_R$ is a squarefree  $R$-module, 
$M:=Rx$ is a squarefree submodule of $\omega_R$. 
Set $$\Upsilon := \{ \, \supp(a) \mid a \in \sM, M_a = [\omega_R]_a \, \} = 
\{ \, \supp(a) \mid a \in \sM, M_a \ne 0 \, \} \subset \cell.$$ 
Here the second equality follows from the condition \eqref{canonical module}.  
It is easy to see that $\s \leq \t \in \Upsilon$ implies $\s \in \Upsilon$. 
So we have a direct sum decomposition 
$\omega_R = M \oplus (\bigoplus_{\supp(a) \in \sM \setminus \Upsilon} [\omega_R]_a)$ as an $R$-module.  
On the other hand, $\omega_R$ is indecomposable. Hence $\omega_R =M \cong R$ as 
$\ZZ \MM$-graded modules.  So we get the sufficiency. 
 
For the second equivalence, it is enough to prove the sufficiency. 
Since $[\omega_R]_0 \cong H^d(X;\kk) \ne 0$, we can take $0 \ne x \in [\omega_R]_0$. 
By argument similar to the above, $(Rx)^+$ is a direct summand of $or_X$. 
Note that $X$ is connected and $\const_X$ is indecomposable. 
Since  $\const_X \cong {\mathcal Ext}^{-d}(or_X, \cDx)$,  
$or_X$ is also indecomposable. Hence $or_X \cong (Rx)^+ \cong \const_X$. We are done.
 \end{proof}

\begin{cor}
The Cohen-Macaulay property and Gorenstein* property of $R$ 
are topological properties of $X$ (while it may depend on $\chara (\kk)$). 
\end{cor}

\begin{proof}
Most of the statement is a direct consequence of Theorems~\ref{CM & Gor}.  
It remains to consider the Gorenstein* property in the case $\dim R=0$.  
Then $R$ is Gorenstein* if and only if $X$ consists of exactly two points. 
So the assertion is clear. 
\end{proof}

\begin{rem}
The main result of Caijun \cite{Cj} is much more general than our  
Theorems~\ref{CM & Gor} (a). However, since he worked in a wider context, 
his argument does not give precise information of local cohomologies and canonical modules.  

Recall that $\MM$ admits a finite subset $\mbra{a_e}_{e \in E}$  
of $\sM$ generating $\fring \MM$ as a $\kk$-algebra. 
Then the polynomial ring $S := \fring{X_e \mid e\in E}$ surjects on $\fring \MM$. 
Let $I_\MM$ be its kernel (i.e., $\kk[\MM] = S/I_\MM$).  
A remarkable result \cite[Theorem~3.8]{BKR} of Bruns et al. 
shows that (if $\MM$ is cone-wise normal) 
there is a generating set $\mbra{a_e}_{e \in E}$ and a term order $\succ$ on $S$ such that 
the initial ideal $\init(I_\MM)$ is a radical monomial ideal. 
In this case, $\init(I_\MM)$ equals to the Stanley-Reisner ring $I_\Delta$ 
of a simplicial complex $\Delta$ which gives a triangulation of $X$.  
%So the geometric realization $|\Delta|$ is homeomorphic to $X$. 
Hence, by a basic fact on Gr\"obner bases, 
the sufficiency of Theorems~\ref{Buchsbaum} and \ref{CM & Gor} (b) follow from 
their result, at least under the additional assumption that $R$ 
admits an $\NN$-grading with $R_0 = \kk$. 
\end{rem}

%
%
%=======================================================
%
%
\end{document}